\documentclass[12pt]{article}
\usepackage{amsmath,amsfonts,amssymb,graphicx,url,a4}

\def\sach{\ |\ }
\def\passe{\lhd}

\def\un{\mathbf{1}}
\renewcommand\leq{\le}
\renewcommand\geq{\ge}

\def\R{{\mathbb R}}

\def\z{{\mathbb Z}}
\def\N{{\mathbb N}}

\def\P{{\mathbb P}}
\def\E{\mathbb{E}}
\def\1{{\mathbb 1}}
\def\F{{\mathcal{F}}}
\def\G{{\mathcal{G}}}

\newtheorem{theoreme}{Theorem}
\newtheorem{conjecture}{Conjecture}[section]
\newtheorem{lemme}[conjecture]{Lemma}
\newtheorem{proposition}[conjecture]{Proposition}
\newtheorem{corollaire}[conjecture]{Corollary}

\newtheorem{remarque}{Remark}

\newtheorem{definition}{Definition}
\newtheorem{notation}{Notation}

\makeatletter
\def\@yproof[#1]{\@proof{ #1}}
\def\@proof#1{\begin{trivlist}\item[]{\em Proof#1.}}
\newenvironment{proof}{\@ifnextchar[{\@yproof}{\@proof{}
}}{~\end{trivlist}} \makeatother

\title{Sufficient conditions of standardness for filtrations of stationary processes taking values in a finite space}
\author{Ga\"el Ceillier}
\date{\today}

\renewcommand{\le}{\leqslant}
\renewcommand{\ge}{\geqslant}

\begin{document}

\maketitle


Soon to be published in The Annals of probability.

\begin{abstract}
Let $X$ be a stationary process with finite state-space $A$. Bressaud et al. recently provided a sufficient condition for the natural filtration of $X$ to be standard when $A$ has size $2$. Their condition involves the conditional laws $p(\cdot|x)$ of $X_0$ conditionally on the whole past $(X_k)_{k \le -1}=x$ and controls the strength of the influence of the ``old'' past of the process on its present $X_0$. It involves the maximal gaps between $p(\cdot|x)$ and $p(\cdot|y)$ for infinite sequences $x$ and $y$ which coincide on their $n$ last terms. In this paper, we first show that a slightly stronger result holds for any finite state-space. Then, we provide sufficient conditions for standardness based on average gaps instead of maximal gaps.
\end{abstract}

\section{Introduction}

\subsection{Setting}

In this paper we study stationary processes $X=(X_n)_{n\in \z}$
indexed by the integer line $\z$ and with values in a finite
set $A$.  We
assume that $X$ is defined recursively as follows: for every $n \in
\z$, $X_{n}$ is a function of the ``past'' $X^{\passe}_{n-1}=(X_k)_{k\le n-1}$ of $X$ and
of a ``fresh'' random variable $U_{n}$, which brings
in some ``new''
randomness.  In particular the process $U=(U_n)_{n\in\z}$ is
independent. To be more specific, we introduce some notations and
definitions about $\sigma$-algebras.

All $\sigma-$fields are assumed to be complete.
For every process $\xi=(\xi_n)_{n\in\z}$ and every $n\in\z$, let $\xi^{\passe}_n=(\xi_k)_{k \le n}$ and
$\F^\xi_n=\sigma(\xi^{\passe}_n)$. The natural filtration of
$\xi$
is the nondecreasing sequence
$\F^\xi=(\F^\xi_n)_{n\in\z}$. Furthermore,
$\F^\xi_\infty=\sigma(\xi_k\,;\,k\in\z)$ and $\F^\xi_{-\infty}$ is
the
tail $\sigma$-algebra $\F^\xi_{-\infty}=\bigcap_{k\in\z}\F^\xi_k$.

We say that a process $U$ is a \textit{governing
  process\/} for $X$, or that $U$ \textit{governs\/} $X$ if, for every
$n \in \z$, (i) $U_{n+1}$ is independent of $\F^{X,U}_n$, and (ii)
$X_{n+1}$ is measurable with respect to $\sigma(U_{n+1})\vee\F^X_n$.
In particular any governing process is independent.
If moreover the $U_n$ are uniform on $[0,1]$, the process $(U,X)$ is -- according to Schachermayer's definition~\cite{schachermayer-certain} and up to a time reversal -- a \textit{parametrization\/} of the process $X$. 

Likewise, we say that a process $U$ is a \textit{generating process\/} for $X$, or that $U$ \textit{generates\/} $X$ if, for every $n \in \z$, $X_n$ is measurable with respect to $\F^U_n$. This is equivalent to the condition that
  $\F_n^X\subset \F_n^U$ for every $n \in \z$, a property which, from now on, we write as $\F^X \subset \F^U$.

One could be led to believe that when $\F^X_{-\infty}$ is trivial, any process governing $X$ generates $X$ as well. But, although notoriously used by Wiener and Kallianpur in~\cite{kallianpur-non} (not published but see a discussion in \cite{masani1966wiener}), this argument is false. As a simple counterexample, assume that $X$ is i.i.d., that every $X_n$ is uniform on $\{-1,1\}$, and set  $U_n=X_n X_{n-1}$ for every $n \in \z$.
Then $\F^X_{-\infty}$ is trivial and $U$ governs $X$ but $X_0$ is independent of $\F^U_\infty$ hence $U$ does not generate $X$.

Governing and generating processes are related to \textit{immersions\/} of filtrations.
Recall that the filtration $\F^X$ is \textit{immersed\/} in the filtration $\F^U$ if $\F^X\subset\F^U$ and if, for every $n \in\z$, $\F^X_{n+1}$ and $\F_n^U$ are independent conditionally on $\F^X_n$. Roughly speaking, this means that $\F_n^U$ gives no further information on $X_{n+1}$ than $\F^X_n$ does.
Equivalently, $\F^X$ is immersed in $\F^U$ if every $\F^X$-martingale is an $\F^U$-martingale. The following easy fact holds (see a proof in section~\ref{S4.2}).

\begin{lemme}
\label{immersion}
If $U$ is a governing and generating process for $X$, then $\F^X$ is immersed in $\F^U$.
\end{lemme}

Another notable property of filtrations is \textit{standardness\/}. Recall that $\F^X$ is \textit{standard\/}
if, modulo an enlargement of the probability space, one can immerse $\F^X$ in a filtration generated by an i.i.d.\ process. Vershik introduced standardness in the context of ergodic theory.
Examples of non-standard filtrations include the filtrations of $[T,T^{-1}]$ transformations, introduced in~\cite{heicklen-t}. Split-word processes, inspired by Vershik's $(r_n)$-adic sequences of decreasing partitions~\cite{vershik1995tds} and studied in~\cite{smorodinsky1998pns} and \cite{laurent2004ftd}, for instance, also provide non-standard filtrations.

Obviously, lemma~\ref{immersion} above implies that if $X$ has a generating and governing process, then $\F^X$ is standard. Whether the converse holds is not known.

Necessary and sufficient conditions for standardness include Vershik's
self-joining criterion and Tsirelson's notion of $I$-cosiness. Both
notions are discussed in~\cite{émery2001vershik} and are based on conditions which are subtle
and not easy to use nor to check in specific cases.

Our goal in this paper is to provide sufficient conditions of standardness that are
easier to use than the ones mentioned above. Each of our conditions involves a measure of the influence of the ``old'' past
of the process on its present. We introduce them in the next section.

\subsection{Statement of the results}

We now introduce some measures of the
influence of the past of a process on its present.
To conveniently state these definitions and, later on, our results, we first introduce some notations.

Recall that $X$ is a stationary process indexed by the integer line $\z$ with values in some finite set
$A$ and with natural filtration $\F^X$.

\begin{notation}

(1) Slabs:
  For any sequence $(\xi_n)_{n\in\z}$ in $A^\z$, deterministic or random, and any integers $i\le
  j$, $\xi_{i:j}$ is the $(j-i+1)$-uple $(\xi_n)_{i\le n\le j}$ in
  $A^{j-i+1}$.
\\
(2) Shifts:
If $k-i=\ell-j$, $\xi_{i:k}=\zeta_{j:\ell}$ means
  that $\xi_{i+n}=\zeta_{j+n}$ for every integer $n$ such that $0\le
  n\le k-i$.
\\
Infinite case:  Let $A^{\passe}$ denote the space of sequences
  $(\xi_n)_{n\le-1}$.
For every $i$ in $\z$, a sequence $(\xi_n)_{n\le i}$ is also considered as an element of $A^{\passe}$ since, similarly to the finite case, one identifies $\xi^{\passe}_{i}=(\xi_n)_{n\le i}$ and $\zeta^{\passe}_{j}=(\zeta_n)_{n\le j}$ if $\xi_{i+n}=\zeta_{j+n}$ for every integer
  $n\le0$.
\\
(3) Concatenation: For all $i \ge 0$, $j \ge 0$, $ x=(x_n)_{1\le n\le i}$ in $A^i$ and $y=(y_n)_{1\le n\le j}$ in $A^j$,
$xy$ denotes the concatenation of $x$ and $y$, defined as
$$
xy=(x_{1},\ldots,,x_{i},y_1,\ldots,y_j),\quad xy\in A^{i+j}.
$$
Infinite case: $i\ge0$, $y=(y_n)_{1\le n\le i}$ in $A^i$ and
$x=(x_n)_{n\le-1}$ in $A^{\passe}$, $xy$ denotes the
concatenation of $x$ and $y$, defined as
$$
xy=(\ldots,x_{-2},x_{-1},y_1,\ldots,y_i),\quad xy\in A^{\passe}.
$$
\end{notation}

\begin{notation}
 For each $n\ge0$, $x \in A^n$ and $a\in A$, set
$$
p(a|x) =\P (X_0 = a\sach X_{-n:-1}=x),
$$
with the convention
$$
p(a|x) = \P(X_0=a) \quad \text{{\rm if }} \P[ X_{-n:-1}=x]=0 .
$$
In the following,
$$
p(\cdot|x) =\P (X_0 = \cdot \sach X^{\passe}_{-1}=x),\quad x \in A^{\passe},
$$
denotes a regular version of the conditional law of $X_0$ given $X_{-1}^{\passe}$.
\end{notation}

We now introduce three quantities $\gamma_n$, $\alpha_n$ and $\delta_n$ measuring the pointwise influence at distance $n$.

\begin{definition}
For every $n\ge0$, let
$$\begin{array}{l}
\gamma_n=1-\inf\left\{\displaystyle\frac{p(a|xz)}{p(a|yz)}\ ;\ a \in A,\ x\in A^{\passe},\ y \in A^{\passe},\ z\in A^n,\ p(a|yz)>0  \right\},\\
~\\
\alpha_n=1-\displaystyle\inf_{z \in A^n} \displaystyle\sum_{a \in A} \inf\left\{p(a|yz)\,;\,y \in A^{\passe}\right\},\\
~\\
\delta_n=\sup\big\{ \| p(\cdot|xz)-p(\cdot|yz) \|\ ;\ x\in A^{\passe},\ y\in A^{\passe},\ z\in A^n\big\},
\end{array}$$
where, for all probabilities $\mu$ and $\nu$ on $A$, $\|\mu-\nu\|$ is the
distance in total variation between $\mu$ and $\nu$, defined as
$$
\|\mu-\nu\|=\frac12\sum_{a \in A} |\mu(a)-\nu(a)|=\sum_{a \in A} \left[\mu(a)-\nu(a)\right]_+.
$$
\end{definition}


Note that the definitions of $\gamma_n$, $\alpha_n$ and $\delta_n$ depend on the choice of the regular version $(p(\cdot|x))_{x \in A^{\passe}}$ of the conditional law of $X_0$ given $X_{-1}^{\passe}$. One needs a ``good" version to get small influences for applying the theorems below.

The sequences $(\gamma_n)_{n\geq 0}$, $(\alpha_n)_{n\geq 0}$ and $(\delta_n)_{n\geq 0}$ are non-increasing, $[0,1]$-valued, and
$\delta_n \le
\gamma_n$, $\delta_n \le\alpha_n$ for every $n \geq 0$ (see the proof in section~\ref{S4.1}).

For every $[0,1]$-valued sequence $(\varepsilon_n)_{n\geq 0}$, we consider the condition
\begin{align}
\tag{$\mathcal{H}(\varepsilon)$}
\sum_{k=0}^{+\infty} \prod_{n=0}^{k}(1-\varepsilon_n)=+\infty.
\end{align}
For instance, $\mathcal{H}(\gamma)$ and $\mathcal{H}(2\delta)$ are respectively
$$
\sum_{k=0}^{+\infty} \prod_{n=0}^{k}(1-\gamma_n)=+\infty,
\quad\mbox{and}\quad
\sum_{k=0}^{+\infty} \prod_{n=0}^{k}(1-2\delta_n)=+\infty.
$$

Observe that if two $[0,1]$-valued sequences $(\varepsilon_n)_{n\geq 0}$ and $(\zeta_n)_{n\geq 0}$ are such that $\varepsilon_n \leq \zeta_n$ for every $n\geq 0$, then $\mathcal{H}(\zeta)$ implies $\mathcal{H}(\varepsilon)$. Hence condition $\mathcal{H}(\varepsilon)$ asserts that $(\varepsilon_n)_{n\geq 0}$ is "small enough" in a way.

The definition of $(\gamma_n)_{n\geq 0}$ and the assumption $\mathcal{H}(\gamma)$ are both stated in
\cite{bressaud2006stationary}. 
The main result of \cite{bressaud2006stationary} is the following.
\begin{theoreme}[X.Bressaud, A.Maass, S.Martinez and J.San Martin's~\cite{bressaud2006stationary}]
\label{theo gamma}
Assume that the size of $A$ is $2$, then $\mathcal{H}(\gamma)$ implies that
 $\F^X$ is standard.
\end{theoreme}

The scope of theorem~\ref{theo gamma} is restricted by the following three conditions. First, the size of $A$ must be $2$. Second, one must control the \textit{ratios\/} of probabilities which define
$\gamma_n$. Third,
$\mathcal{H}(\gamma)$ implies that $\gamma_0 < 1$,
therefore one can show that $\mathcal{H}(\gamma)$ implies the existence of $c>0$ such that $p(a|x)\ge c$ for every $x$ in $A^{\passe}$ and $a$ in $A$ such that $\P[X_0=a]>0$ (see the proof in section~\ref{S4.4}).


Our first result allows to get rid of the first two restrictions.

\begin{theoreme}\label{theo delta}
  (1) Assume that $A$ is finite, that $2\delta_0<1$
and that
  $\mathcal{H}(2\delta)$ holds. Then $\F^X$ is standard.
\\
  (2) If the size of $A$ is $2$,  $\mathcal{H}(\delta)$ alone implies that $\F^X$ is standard.
\end{theoreme}

Theorem~\ref{theo delta} generalizes and improves on theorem~\ref{theo
  gamma} of \cite{bressaud2006stationary}, since $\delta_n \le
\gamma_n$ for every $n$.
Note that the straight adaptation of the proof of~\cite{bressaud2006stationary} to sizes of $A$ at least $3$
leads to the more stringent condition $\mathcal{H}(2\gamma)$.

Another measure of influence, based on the quantities $\alpha_n$ defined before, is introduced and used in~\cite{comets2002processes} (actually the notation there is $a_n=1-\alpha_n$). The authors show that if $\mathcal{H}(\alpha)$ holds, there exists a perfect sampling algorithm for the process $X$, a result which implies that $\F^X$ is standard.
But since $\delta_n \leq \alpha_n$ for every $n \geq 0$, the result of~\cite{comets2002processes} does not imply theorem~\ref{theo gamma}.

Theorems~\ref{theo gamma} and
\ref{theo delta} and the exact sampling algorithm of
\cite{comets2002processes} all require an upper bound of some pointwise
influence sequence.
Our next result uses a less restrictive
hypothesis based on some average influences $\eta_n$, defined below.


\begin{definition}
For every $n\ge0$, let $\eta_n$ denote the average influence at distance $n$, defined as
$$
\eta_n=\sum_{z \in A^{n}} \E \big[ \| p(\cdot|z)-p(\cdot|X^{\passe}_{-n-1}z) \|\big]\cdot\P[X_{-n:-1}=z],
$$
and call $\mathcal{H}'(\eta)$ the condition
\begin{align}
\tag{$\mathcal{H'}(\eta)$}
\sum_{k=0}^{+\infty} \eta_k<+\infty.
\end{align}
\end{definition}

Note that $\eta_n$ is also
$$
\eta_n= \E \big[\|p(\cdot|Y_{-n:-1})-p(\cdot|X^{\passe}_{-n-1}Y_{-n:-1})\|\big],
$$
where $Y$ is an independent copy of $X$.

\begin{definition}{\bf (Priming condition)}\label{priming condition}\\
We say that the process $X$ fulfills the priming condition if
for every $a$ in $A$,
  $p(a|X^{\passe}_{-1})>0$ almost surely.
\end{definition}

\begin{theoreme}\label{theo e}
  Assume that $A$ is finite and that $X$ fulfills the priming condition. Then, $\mathcal{H}'(\eta)$
  implies that $\F^X$ is standard.
\end{theoreme}

The sequence $(\eta_n)_{n \geq 0}$ is $[0,1]$-valued.
If $\eta_n<1$ for every $n \leq 0$, then $\mathcal{H}'(\eta)$
clearly implies $\mathcal{H}(\eta)$.
Yet, since $\eta_n\le \delta_n$ for every $n \geq 0$ (see the proof in section~\ref{S4.1}), the condition $\mathcal{H}'(\eta)$ cannot be compared to the conditions $\mathcal{H}(\delta)$ and $\mathcal{H}(2\delta)$.

Theorem~\ref{theo e} gives a remarkable result for chains with memory of variable length. These chains, studied notably in~\cite{galves2008stochastic} and widely used for mathematical models, are stationary processes $X$ taking values in a finite alphabet $A$, such that the distribution of $X_0$ given the past $X^{\passe}_{-1}$ depends only on a past $X_{-\ell:-1}$ of length $\ell$, where $\ell$ is random and measurable with respect to $\F^X_{-1}$.

More precisely, for $x \in A^\passe$, let
\begin{eqnarray*}
\ell(x)&=&\inf\{n \geq 0\,;\,y\mapsto p(\cdot|yx_{-n:-1})\ \mbox{is constant on}\ A^\passe\}\\
&=&\inf\{n \geq 0\,;\,\forall y \in A^\passe,\,p(\cdot|yx_{-n:-1})=p(\cdot|x)\}.
\end{eqnarray*}

Then $X$ is a variable length Markov chain if $\ell(X^\passe_{-1})$ is almost surely finite.
The following result holds.

\begin{corollaire}\label{memoire variable}
If $X$ fulfills the priming condition and if $\ell(X^\passe_{-1})$ is integrable, then the natural filtration $\F^X$ is standard.
\end{corollaire}
Once again we refer the reader to section~\ref{S4.3} for the proof.

Here is a plan of the rest of the paper. In section \ref{S1}, we prove
theorem \ref{theo delta}. In section \ref{S2}, we prove
theorem \ref{theo e}. In section~\ref{S3}, we  compare
theorems \ref{theo delta} and \ref{theo e} through examples. Finally in section~\ref{ss.profs},
we prove some facts stated without proof in the introduction, namely lemma~\ref{immersion}, corollary~\ref{memoire variable}, a consequence of the assumption $\mathcal{H}(\gamma)$ and some inequalities involving the quantities $\alpha_n,$ $\gamma_n$, $\delta_n$ and $\eta_n$.
\section{Pointwise influence}\label{S1}


\subsection{Construction of a governing sequence}\label{S1.1}

We construct a governing sequence with values in the standard simplex on $\#A$ vertices.

\begin{notation}
Let $H$ be the hyperplane in $\R^A$ defined by
$$
H=\{x=(x_a)_{a\in A} \in \R^A:\sum_{a\in A}x_a=1\}.
$$
Let $S$ be the simplex in $H$ defined by
$$
S:=(\R_+)^A\cap H = \left\{x \in (\R_+)^A\,:\,\sum_{a\in A}x_a=1 \right\}.
$$
In other words, $S={\rm Conv}(E_A)$ is the convex enveloppe of the canonical basis $E_A=(E_a)_{a\in A}$ of $\R^A$.
\\
Let $\lambda$ denote the Lebesgue measure on $H$ and $\mu=(\un_{S}/\lambda(S))\lambda$ the uniform distribution on $S$.
\end{notation}

\begin{notation}
For any probability $p$ on $A$, let
$$
G(p)=\big(p(a)\big)_{a\in A}=\sum_{a\in A}p(a)\,E_a,\quad G(p)\in S.
$$
For $a$ in $A$, denote by $f_a(\cdot ,p)$ the affine map from $H$ to $H$ which sends
$E_a$ on $G(p)$ and lets invariant $E_b$ for every $b$ in $A$, $b\ne a$.
Let
$$
S_a(p)=f_a(S,p)={\rm Conv}\Big(\{G(p)\}\cup E_A\setminus\{E_a\}\Big).
$$
\end{notation}

A short computation yields the interpretation of $p(a)$ below.

\begin{lemme}
\label{ll.det}
For any $a$ in $A$, $\det (f_a(\cdot,p))=p(a)$. Therefore, for any measurable $B\subset S$,
$$
\lambda[f_a(B,p)]=\lambda[B]\,p(a).
$$
In particular $\lambda(S_a(p))=\lambda(S)\,p(a)$, hence $\mu(S_a(p))=p(a)$.
\end{lemme}

We now characterize $S_a(p)$.

By convention, for every $r>0$, we set $r/0=\infty$, and $0/0=0$.
\begin{lemme} \label{S_a=sigma}
Let $p$ be a probability on $A$, $a$ in $A$, then
$$
S_a(p)=\left\{x=(x_a)_{a \in A}\in S\,:\,\frac{x_a}{p(a)}=\min_{b\in A}\frac{x_b}{p(b)}\right\}=\left\{x\in S\,:\,\forall bb \in A, \frac{x_a}{p(a)}\leq{b\in A}\frac{x_b}{p(b)}\right\}.
$$

\end{lemme}

\begin{corollaire} $S$ is the union of the simplices $S_a(p)$, with $a$ in $A$ and that, if $a\ne b$, the simplices $S_a(p)$ and $S_b(p)$ meet only at their boundary.
\end{corollaire} Proof : it's a straight corollary of the lemmas~\ref{ll.det} and~\ref{S_a=sigma}.

\begin{proof}
Call $\Sigma_a(p)$ the right-hand side. Since $\Sigma_a(p)$ is a convex polyedron and contains the points $G(p)$ and $E_b$ for every $b\ne a$,
$S_a(p) \subset \Sigma_a(p)$.

As regards the other inclusion, let $x=(x_a)_{a\in A}$ in $\Sigma_a(p)$. Then $x_a/p(a)$ is finite, and
$$
x=\frac{x_a}{p(a)} G(p)+\sum_{b \ne a} \left(x_b-p(b)\frac{x_a}{p(a)}\right)E_b.
$$
From the definition of $\Sigma_a(p)$, $x_b/p(b)\ge x_a/p(a)$ for every $b\ne a$, hence one has $x_b-p(b)x_a/p(a)\ge0$ for every $b\ne a$.
Furthermore,
$$
\frac{x_a}{p(a)}+\sum_{b \ne a}\left( x_b-p(b)\frac{x_a}{p(a)}\right)=\frac{x_a}{p(a)}+\sum_{b\in A}\left( x_b-p(b)\frac{x_a}{p(a)}\right)
=\sum_{b\in A}x_b=1,
$$
hence $x$ is indeed a barycenter of the points $G(p)$ and $E_b$ for $b\ne a$. This concludes the proof.
\hfill $\square$ \end{proof}

One knows that the simplices $(S_a(p))_{a\in A}$ cover $S$ and intersect only on a set of measure zero. Hence, for almost every $s$ in $S$, there exists a unique $a$ in $A$ such that $s \in S_a(p)$. Our next definition deals with the tie cases.

\begin{definition}
Fix once and for all a total ordering of $A$. For every $s$ in $S$ and every probability $p$ on $A$ with full support, define
$$
g(s,p)=\min\{a\in A\,:\,s \in S_a(p)\}.
$$
\end{definition}

\eject
\newpage

\begin{figure}[h]
\begin{center}
\includegraphics[width=.95\textwidth]{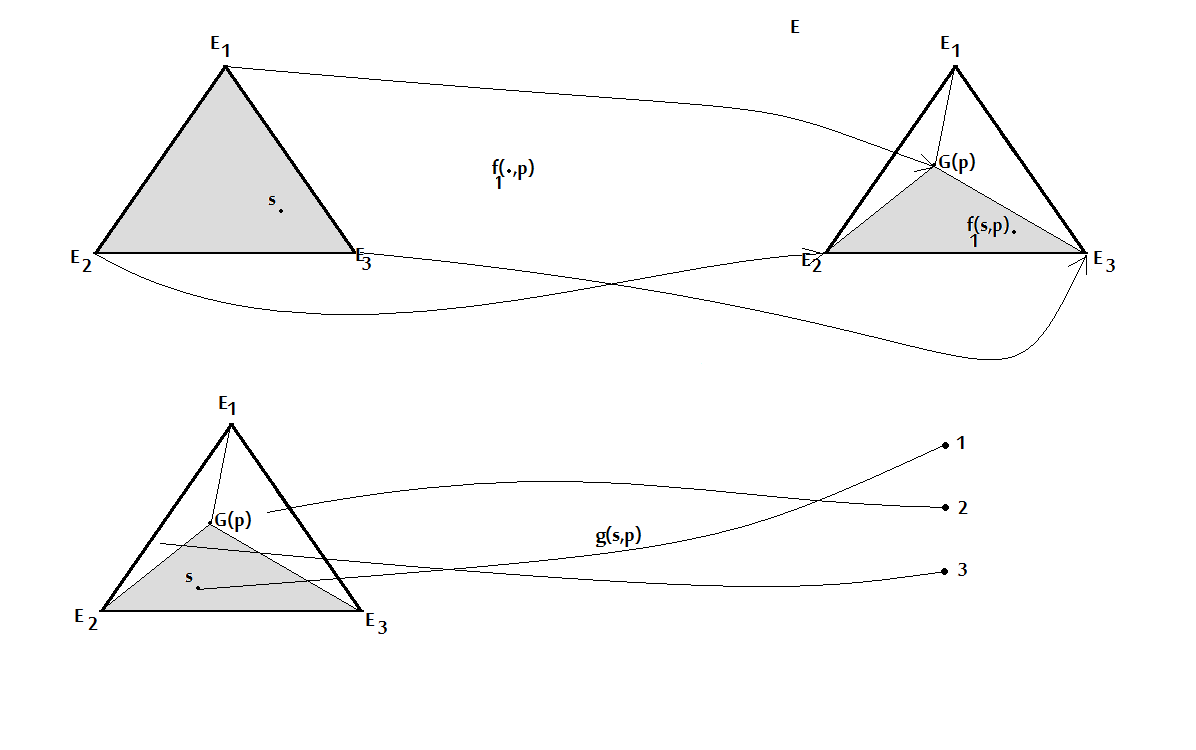}
\caption{$f(\cdot,p)$ and $g(\cdot,p)$}
\end{center}
\label{Graphique innovations}
\end{figure}

\begin{lemme}
Let $U$ denote a random variable uniformly distributed on $S$. Then the distribution of $g(U,p)$ is $p$.
\end{lemme}

Indeed, up to negligible events, $\{g(U,p)=a\}=\{U\in S_a(p)\}$, hence
$$
\P[g(U,p)=a]=\mu(S_a(p))=p(a).
$$

The following lemma is our main tool to construct governing sequences.

\begin{lemme}\label{innovation}
Let $X$ be a random variable with distribution $p$ on $A$. Let $W$ be a random variable with uniform distribution on $S$ and independent of $X$. Introduce
$$
U=f_X(W,p)=\sum_{a \in A} f_a(W,p) \mathbf{1}_{\{X=a\}}.
$$
Then $U$ is uniformly distributed on $S$ and $X=g(U,p)$ almost surely.
\end{lemme}

\begin{proof}
Since $U\in f_X(S,p)=S_X(p)$, $X=g(U,p)$ almost surely.
We now prove that $U$ is uniformly distributed on $S$.

The sets $S_a(p)$ for $a$ in $A$ cover $S$  and their pairwise intersections are negligible for $\lambda$. Hence, for every Borel subset $B$ of $S$,
\begin{align*}
  \P [U\in B] &= \sum_{a\in A} \P[X=a\ ;\ f_a(W,p) \in B ]
   = \sum_{a\in A} \P[X=a]\cdot\P[W \in f_a(\cdot,p)^{-1}(B)]\\
   &= \sum_{a\in A} p(a)\,\frac{\lambda\big(f_a(\cdot,p)^{-1}(B)\cap S\big)}{\lambda(S)}
   = \sum_{a\in A} \frac{\lambda\big(B \cap S_a\big)}{\lambda(S)}= \frac{\lambda(B)}{\lambda(S)}= \mu(B),
\end{align*}
where the second equality stems from the independence of $X$ and $W$ and the fourth equality stems from lemma~\ref{ll.det}. This concludes the proof. \hfill $\square$ \end{proof}

\subsection{Upper bound of the error}\label{S1.2}

In this section we study the dependence of the random variable $g(U,p)$ with respect to $p$. The following result will be used twice.

\begin{proposition}[Upper bound of the error]
\label{prop maj erreur}
Let $U$ be a random variable uniformly distributed on $S$.
Let $p$ and $q$ be two probabilities on $A$. Then,
$$
\P[g(U,p) \ne g(U,q)] \leq 2\,\|p-q\|.
$$
In the special case $\#A=2$,
$$
\P[g(U,p) \ne g(U,q)] = \|p-q\|.
$$
\end{proposition}


\begin{remarque}The better result when $\#A=2$ is the reason why
theorem~\ref{theo delta} involves weaker hypotheses on
$(\delta_n)_n$ in this case.
\end{remarque}

\eject
\newpage

\begin{figure}[h]
\begin{center}
\includegraphics[width=.95\textwidth]{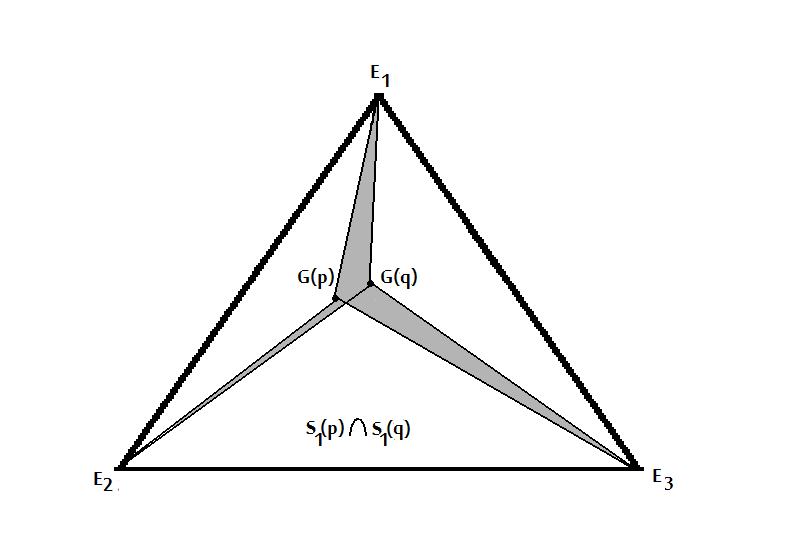}
\caption{Computation of $\P[g(U,p) \ne g(U,q)]$. The grey area shows the $s \in S$ such that $g(s,p) \ne g(s,q)$.}
\end{center}
\label{Graphique innovations 2}
\end{figure}

\begin{proof}[of proposition \ref{prop maj erreur}]
Assume without loss of generality that $U$ is constructed from i.i.d.\ random variables $(\varepsilon_a)_{a\in A}$ exponentially distributed with parameter $1$, as
follows. For every $a$ in $A$,
$$
U_a=\frac{\varepsilon_a}{\sum_{b\in A}\varepsilon_b}.
$$
The event $\{g(U,p) \ne g(U,q)\}$ depends on $(\varepsilon_a)_a$, as follows.
By definition of $g$, up to negligible events,
$$
\{g(U,p) = g(U,q)\} = \bigcup_{a\in A}C_a,\quad\text{with } C_a= \{U \in S_a(p) \cap S_a(q)\}.
$$
Furthermore, since for every $a\in A$, $\P[C_a]=0$ if $p(a)=0$ or $q(a)=0$, and since $\mu(C_a \cap C_b)=0$ for $a \ne b$, one gets
$$
\P[g(U,p) = g(U,q)]=\sum_{a \in A} \P[C_a] {\bf 1}_{\{p(a)>0,\ q(a)>0\}}.
$$
For every $a \in A$ such that  $p(a)>0$ and $q(a)>0$, lemma~\ref{S_a=sigma} gives
$$
C_a=\left\{\frac{\varepsilon_a}{p(a)}=\min_b \frac{\varepsilon_b}{p(b)}\,;\,\frac{\varepsilon_a}{q(a)}=\min_b\frac{\varepsilon_b}{q(b)}\right\}
= \left\{ \varepsilon_a \leq \min_b\Big(p(a)\frac{\varepsilon_b}{p(b)}\,,\, q(a)\frac{\varepsilon_b}{q(b)}\Big)\right\},
$$
hence
$$
C_a=
\bigcap_{b\ne a}\left\{ \varepsilon_b \ge\lambda_{b/a}\varepsilon_a\right\},
\quad
\lambda_{b/a}=\max\Big(\frac{p(b)}{p(a)},\,\frac{q(b)}{q(a)}\Big).
$$
Conditioning on $\varepsilon_a$ and using that the random variables $(\varepsilon_b)_{b\ne a}$ are i.i.d., exponentially distributed and independent of $\varepsilon_a$, one gets
$$
\P\Big[C_a\,\big|\,\varepsilon_a\Big]
=
\P\Big[\bigcap_{b\ne a} \left\{ \varepsilon_b \ge \lambda_{b/a}\varepsilon_a\right\} \Big| \varepsilon_a\Big]
=
\prod_{b \ne a} \exp\left(- \lambda_{b/a} \varepsilon_a\right),
$$
hence
$$
\P\Big[C_a\Big]
=
\E\left(\exp\left(- \Big(\sum_{b \ne a} \lambda_{b/a}\Big) \varepsilon_a\right)\right)
=
\frac{1}{1+\sum_{b \ne a} \lambda_{b/a}}.
$$
Therefore
$$
\P[g(U,p) = g(U,q)]
=
\sum_{a\in A}\P\left[C_a\right] {\bf 1}_{\{p(a)>0,\ q(a)>0\}}
=
\sum_{a\in A} \frac{{\bf 1}_{\{p(a)>0,\ q(a)>0\}}}{1+\sum_{b \ne a} \lambda_{b/a}}.
$$
This last expression is not so easy to compute because each $\lambda_{b/a}$ is defined as a maximum. However,
$$
\sum_{a \in A}\frac{{\bf 1}_{\{p(a)>0\}}}{1+\sum_{b \ne a} p(b)/p(a)}=\sum_{a \in A }\frac{p(a) {\bf 1}_{\{p(a)>0\}}}{p(a)+\sum_{b \ne a} p(b)}=\sum_{a}p(a)=1.
$$
Subtracting the expression for $\P[g(U,p) = g(U,q)]$ to this, one gets
\begin{align*}
\P[g(U,p) \ne g(U,q)]&=
\sum_{a \in A}\frac{{\bf 1}_{\{p(a)>0\}}}{1+\sum_{b \ne a}p(b)/p(a)}-\sum_{a \in A}\frac{{\bf 1}_{\{p(a)>0,\ q(a)>0\}}}{1+\sum_{b \ne a} \lambda_{b/a}}.
\end{align*}
Coming back to the definition of $\lambda_{b/a}$ and using simple algebraic manipulations, one gets for any $a\in A$ such that $p(a)>0$ and $q(a)>0$,
\begin{align*}
\frac{1}{1+\sum_{b \ne a}p(b)/p(a)}-\frac{1}{1+\sum_{b \ne a} \lambda_{b/a}}
&=\frac{\sum_{b\ne a}\left(\lambda_{b/a}-p(b)/p(a)\right)}{\Big(1+\sum_{b \ne a} p(b)/p(a)\Big)\Big(1+\sum_{b \ne a} \lambda_{b/a}\Big)}\\
&=p(a)\frac{r(a)}{q(a)+r(a)},
\end{align*}
where
$$
r(a)=\sum_{b}\left[ q(b)p(a)-p(b)q(a) \right]_+.
$$
Furthermore, for any $a\in A$ such that $p(a)>0$ and $q(a)=0$ one gets
$$
\frac{1}{1+\sum_{b \ne a}p(b)/p(a)}=p(a)=p(a)\frac{r(a)}{q(a)+r(a)}.
$$

Summing on every $a$, one gets finally
$$
\P[g(U,p) \ne g(U,q)]
=
\sum_{a \in A}p(a)\frac{r(a)}{q(a)+r(a)},
$$

If $A=\{a,a'\}$ and, for example, $q(a)<p(a)$, then $r(a)=p(a)-q(a)$ and $r(a')=0$, hence $$
\P[g(U,p) \ne g(U,q)]
=r(a)=p(a)-q(a)=\|p-q\|.
$$
In the general case, note that
$$
q(a)+r(a)\ge q(a)+\sum_{b}\left(q(b)p(a)-p(b)q(a)\right)=p(a),
$$
hence
\begin{align*}
\P[g(U,p) \ne g(U,q)]
&\le\sum_{a \in A}r(a)=
\sum_{a \in A}\sum_{b\ne a} \left[p(a)q(b)-q(a)p(b)\right]_+
\\
&\le
\sum_{a \in A}\sum_{b\ne a}p(a)\Big[\left[ q(b)-p(b) \right]_++p(b)\left[ p(a)-q(a) \right]_+\Big],
\end{align*}
where the last inequality stems from the fact that $(u+v)_+\le (u)_++(v)_+$ for every $u$ and $v$. Finally, the last double sum is at most $2\|p-q\|$,
which ends the proof in the general case.
\hfill $\square$ \end{proof}


Recall that if $p$ and $q$ are two {\it fixed} probabilities on $A$, then for every random variables $Z_p$ and $Z_q$ with laws $p$ and $q$ defined on the same probability space,
$$
\P[Z_p\ne Z_q]\geq\|p-q\|.
$$
Conversely, a standard construction in coupling theory provides random variables $Z_p$ and $Z_q$ with laws $p$ and $q$ such that
$
\P[Z_p\ne Z_q]=\|p-q\|.
$

The interest of proposition~\ref{prop maj erreur} is to provide a global coupling of {\it all} probabilities on $A$. One can wonder whether the constant $2$ in this proposition can be improved. Our next result (not used in the sequel) shows that the constant $2$ is optimal for the coupling $(g(U,p))_p$, and that it is not possible to do much better with any other global coupling.

\begin{proposition}[Optimality of the upper bound of the error]\label{optimalite lemme erreur}~\\
If $\#A\ge 3$, the constant $2$ in the inequality $\P[g(U,p) \ne g(U,q)] \le 2\|p-q\|$ of proposition~\ref{prop maj erreur} is optimal.
\\
Furthermore, if $(Z_p)_p$ is a family of random variables indexed by probabilities on $A$, where each $Z_p$ follows the law $p$, then there exist two probabilities $p\ne q$ such that
$$
\P[Z_p\ne Z_q] \ge2(1-1/\#A)\|p-q\|.
$$
\end{proposition}


\begin{proof} The first part of the proposition follows from the explicit example where $\{a,b,c\}\subset A$, $p(a)=q(a)=1-\varepsilon$ and $p(b)=q(c)=\varepsilon$ in the limit $\varepsilon\to0$.

As regards the second part, let $N=\#A$ and for every $a$ in $A$, let $Z^a$ denote the random variable of $(Z_p)_p$ with uniform distribution on $A \setminus \{a\}$. Choose $a_0 \in A$, and consider the random set $D$ of the elements $a$ of $A$ such that $Z^a = Z^{a_0}$.
For every $a,b\in A$,
$$
\mathbf{1}_{[Z^a \ne Z^b]}\geq\mathbf{1}_{[a\in D, b\notin D]}+\mathbf{1}_{[a \notin D, b \in D]}.
$$
By summing over $a,b \in A$ and by taking expectations, one gets
$$
\sum_{a,b}\P[Z^a \ne Z^b]
\geq
\E\big[2\#D(N-\#D)\big].
$$
Of course $a_0\in D$, and $Z^{a_0}\notin D$, since $Z^a\ne a$ almost surely for every $a$. Thus  $1 \le \#D\le N-1$.
Hence,
$$
\sum_{a,b}\P[Z^a \ne Z^b]
\ge
2(N-1).
$$
There are at most $N(N-1)$ nonzero terms in the sum above, hence there exist $a\ne b$ such that
$$
\P[Z^a \ne Z^b]
\ge
2/N.
$$
Since $\|p_a-p_b\|=1/(N-1)$, this yields
$\P[Z^a \ne Z^b]\ge2(1-1/N)\|p_a-p_b\|$, which ends the proof.
\hfill $\square$ \end{proof}

\subsection{Proof of theorem \ref{theo delta}}\label{S1.3}

Let $W=(W_n)_{n \in \z}$ be an i.i.d.\ sequence of random variables,  uniformly distributed on $S$, independent of the process $X$. Applying the construction of governing sequences in section~\ref{S1.1}, we introduce, for every $n$ in $\z$,
$$
U_n=f_{X_n}(W_n,P_{n-1})\quad\text{with}\quad
P_{n-1}= p(\cdot|X^\passe_{n-1}).
$$
Let $n \in \z$. Thanks to the stationarity of the process $X$ and to the independence of $X$ and $W$, $P_{n}$ is the conditional law of $X_{n+1}$ given $\F^{X,W}_{n}$. Since $W_{n+1}$ is independent of $\F^{X,W}_{n}$ and $X_{n+1}$, lemma~\ref{innovation} yields
that
\begin{enumerate}
\item $U_{n+1}$ is independent of $\F^{X,W}_{n}$, and therefore of $\F^{X,U}_{n}$,
\item $U_{n+1}$ is uniformly distributed on $S$,
\item $X_{n+1}=g(U_{n+1},P_{n})$ almost surely.
\end{enumerate}

For every $T$ in $\z$, we now define a process $X^T$ which is a function of $(U_n)_{n\ge T+1}$ in such a way that $X^T$ approximates $X$ when $T\to-\infty$.

Let $X^T_n=a_0$ for $n\le T$ with $a_0\in A$ fixed, and assume that $X^T_n$ is defined up to time $n\ge T$. Define
$$
X^T_{n+1}=g(U_{n+1}, P^T_n)\quad\text{where}\quad
P^T_{n}= p(\cdot|(X^T)^\passe_n).
$$
Proposition \ref{prop maj erreur} implies that for $n \ge T$,
$$
\P\big[X_{n+1}\ne X^T_{n+1}\sach\F^{X,U}_{n}\big]
=
\P\big[g(U_{n+1},P_n)\ne g(U_{n+1},P^T_{n})\sach\F^{X,U}_{n}\big]
\leq
2\| P_n-P^T_n\|,
$$
because $P_n$ and $P^T_{n}$ are measurable for $\F^{X,U}_{n}$ and $U_{n+1}$ is independent of $\F^{X,U}_{n}$.

For $n$  in $\z$, let $L_n^{T}$ count the number of consecutive times before $n$ such that $X^T$ and $X$ coincide, that is,
$$
L_n^{T}=\max\{k \ge 0\,:\,X^T_{n-k+1:n}=X_{n-k+1:n}\}.
$$
On the event $\{L_n^{T}=\ell\}$, the sequences $X^\passe_{n}$ and $(X^T)^\passe_n$ coincide on their last $\ell$ terms. Hence, on the event $\{L_n^{T}=\ell\}$,
$$
\| P_n-P^T_n\|
\le
\sup \Big\{\|p(\cdot|xz)-p(\cdot|yz)\|\,;\, x\in A^\passe,\,y\in A^\passe,\,z\in A^\ell \Big\}
  = \delta_\ell .
$$
One gets
$$
\P\big[X_{n+1}\ne X^T_{n+1}\sach\F_{n}^{X,U}\big]
\le
2\delta_{L_n^{T}}.
$$
The end of our proof follows the method in~\cite{bressaud2006stationary}: consider a $\z^+$-valued Markov chain, $Z=(Z_n)_{n\ge 0}$ starting from $Z_0=0$, with transition probabilities
$$
p_{i,i+1}=1-2\delta_i,\quad
p_{i,0}=2\delta_i, \text{ for every }i \geq 0.
$$
For any $n\ge T$, it happens that $L_n^{T}$ dominates stochastically $Z_{n-T}$, in the sense of the following lemma.

\begin{lemme}
\label{l.stochdom}
 For every $k \ge 0$ and $n\ge T$, $\P[L_n^{T}\ge k]\ge \P[Z_{n-T}\ge k]$.
\end{lemme}

\begin{proof}[of lemma \ref{l.stochdom}]
The result is obvious for $n=T$ since $Z_0=0$.
Assume that the result holds for $n\ge T$. Then,
\begin{eqnarray*}
    \P[L_{n+1}^{T}\ge k+1]&=&\P[L_n^{T}\ge k,\, X_{n+1}=X_{n+1}^{T}] \\
    &=&\E \big[ \un_{\{L_n^{T}\ge k\}}\P[ X_{n+1}=X_{n+1}^{T}\sach\F^{X,U}_n  ]\big]\\
    &\ge& \E \big[ \un_{\{L_n^{T}\ge k\}}(1-2\delta_{L_{n}^T})\big].
\end{eqnarray*}
Since $(\delta_i)_i$ is non-increasing and since $2\delta_i<1$ for every $i$,
the sequence indexed by $i$ of general term $\un_{\{i\ge k\}}(1-2\delta_i)$ is non-decreasing.
By induction, one obtains
\begin{eqnarray*}
    \P[L_{n+1}^{T}\ge k+1]&\ge& \E \big[ \un_{\{Z_{n-T}\ge k\}}(1-2\delta_{Z_{n-T}})\big]\\
    &=&\P[Z_{n-T}\ge k,\, Z_{n-T+1}=Z_{n-T}+1]\\
    &=&\P[Z_{n-T+1}\ge k+1],
\end{eqnarray*}
which ends the recurrence over $n\ge T$ and the proof of the lemma.
\hfill $\square$ \end{proof}

Using this to estimate $\P[L_{n}^{T} = 0]$, one gets
$$
\P\big[X^T_n\ne X_n\big] = \P[L_{n}^{T} = 0] \le \P[Z_{n-T} = 0].
$$
Let $\mu$ be the measure defined on $\z^+$ by
$$
\nu(k)=\prod_{n=0}^k(1-2\delta_n),
$$
for every $k\geq 0$. the hypothesis of theorem~\ref{theo delta} ensure that $\mu$ has infinite mass.

If   $\prod_{n=0}^{+\infty}(1-2\delta_n)=0$, then $\mu$ is invariant and the state $0$ is recurent. But the chain $Z$ is irreducible since the hypothesis of theorem~\ref{theo delta} forces the positivity of the probabilities $(1-2\delta_n)$. Hence $Z$ is null recurrent.\\
If  $\prod_{n=0}^{+\infty}(1-2\delta_n)>0$, then $Z$ is transient since for every $i$, the probability of never returning to $i$ from $i$ is
 $\prod_{n=i}^{+\infty}(1-2\delta_n)>0$.
In both cases
$$
\P\big[X^T_n \ne X_n\big] \to 0 \text{ when } T \to - \infty.
$$
In other words, $X^T_{n}$ converges in probability to $X_n$ when $T\to-\infty$,
in particular $X_{n}$ is measurable for $ \F^U_{n}$, which proves that $U$ generate $X$. Using lemma~\ref{immersion} one gets that the filtration $\F^X$ is immersed in $\F^U$,  therefore $\F^X$ is standard.
This ends the proof of theorem \ref{theo delta}.

\section{Average influences}\label{S2}
This section is devoted to the proof of theorem~\ref{theo e}.


\subsection{Priming lemma}\label{S2.2}

Recall that the governing sequence $U$ with values in $S$ and based on lemma~\ref{innovation} is defined by
$$
U_n=f_{X_n}(W_n,P_{n-1})\quad\text{where}\quad P_{n-1}=p(\cdot|X^\passe_{n-1}),
$$
where $W=(W_n)_{n \in \z}$ is a sequence of i.i.d. random variables uniform on $S$, independent of $X$. Recall also that, from lemma~\ref{innovation}, $X_n=g(U_n,P_{n-1})$ almost surely for every $n \geq 0$.

Let $\ell >0$. Let us show that with probability close to $1$, for each $x$ in $A^\ell$,
$X_{1:\ell}=x$ as soon as $U_{1:\ell}\in B_x$ where $B_x$ is a measurable subset of $S^\ell$ with $\mu$-measure independent of $x$.

Recall that $X$ satisfies the priming condition if for every $a$ in $A$, $p(a|X^\passe_{-1})>0$ almost surely.

\begin{lemme}{\bf(Priming lemma)}\label{lemme amorce}
Set $\ell > 0$.
If $X$ verifies the priming condition,
then for every $\varepsilon \in ]0,1[$, there exist a real number $\beta_\ell>0$ and a collection $(B_x)_{x \in A^\ell}$ of Borel sets of $S^\ell$ such that for every $x \in A^\ell$,
$$\mu^{\otimes \ell}[B_x]=\beta_\ell ~\text{ {\rm and }}~
\P \big[X_{1:\ell}=x~|~U_{1:\ell}\in B_x\big]\geq 1-\varepsilon.$$

Therefore if $Y$ is a random variable valued in $A^\ell$ independent of $(X_n,U_n)_{n \in \z}$,
$$\P[U_{1:\ell}\in B_Y]=\beta_\ell ~\text{ {\rm and }}~ \P\big[X_{1:\ell}=Y~|~U_{1:\ell}\in B_Y\big]\geq 1-\varepsilon.$$
\end{lemme}

\begin{proof}
For every $n\in \z$, let $P_n=p(\cdot|X^\passe_{n})$. Thanks to the stationarity of the process $X$, $P_n$ is the conditional law of $X_{n+1}$ given $\F^X_n$ and the priming condition ensures that the support of $P_n$ is $A$ almost surely.

Let $\varepsilon \in ]0,1[$ and $\ell >0$. For any fixed $x \in A^\ell$ let us construct by induction Borel sets $B_1^{x},...,B_\ell^{x}$ of $S$ with positive measure such that for every $m\in \{1,...,\ell\}$,
$$\P(C_m)\geq\Big(1-\frac{\varepsilon}{\ell}\Big)~ \mu(B^x_m) ~\P(C_{m-1})>0, $$
where $C_0=\Omega$ and for every $m \in \{1,...,\ell\}$,
$$C_m=\big\{X_{1:m}=x_{1:m}~;~U_{1:m} \in B^{x}_1\times ...\times B^{x}_m\big\}.$$
Let $m\in \{1,...,\ell\}$. Assume that $B_1^{x},...,B_{m-1}^{x}$ are constructed verifying the induction hypothesis. Since
$\P[C_{m-1}]>0,$
one gets, thanks to the priming condition,
$$\P\big[P_{m-1}(x_m)=0~\big|~C_{m-1}\big]=0.$$
Therefore one can choose a real number $q \in ]0,1]$ such that
$$\P\big[P_{m-1}(x_m)\leq q~\big|~C_{m-1}\big]<\frac{\varepsilon}{\ell}.$$

Set $U_m=(U_{m,1},...,U_{m,N})$. Since $X_m=g(U_m,P_{m-1})$ almost surely, one gets up to negligible events,
\begin{eqnarray*}
\{X_m=x_m\}&\supset&\{X_m=x_m~;~P_{m-1}(x_m)>q  \}\\
&=&\big\{g(U_m,P_{m-1})=x_m~;~P_{m-1}(x_m)>q  \big\}\\
&=&\left\{\frac{U_{m,x_m}}{P_{m-1}(x_m)}=\min_{k \in A}\frac{U_{m,k}}{P_{m-1}(k)}~;~P_{m-1}(x_m)>q  \right\}\\
&\supset&\left\{\frac{U_{m,x_m}}{q}\leq \min_{k \ne x_m}U_{m,k}~;~P_{m-1}(x_m)>q  \right\}.\\
\end{eqnarray*}
Set $$B^{x}_m=\Big\{(y_1,...,y_N)\in S~;~\frac{y_{x_m}}{q}\leq \min_{k\ne x_m} y_k\Big\}.$$
Then $\mu(B^{x}_m)>0$ and $$\big\{X_m=x_m~;~U_m\in B^x_m\big\}\supset\big\{U_m\in B^x_m~;~ P_{m-1}(x_m)>q\big\}.$$
Since $$C_m=\{X_m=x_m~;~U_m\in B^x_m\}\cap C_{m-1},$$ the independence of $U_m$ and $\F^{X,U}_{m-1}$ and the choice of $q$ yield
\begin{eqnarray*}
\P[C_m]&\geq&\P\big[U_m\in B^{x}_m~;~P_{m-1}(x_m)> q~;~C_{m-1}\big]\\
&=&\mu(B^{x}_m)\P\big[P_{m-1}(x_m)> q~;~C_{m-1}\big]\\
&\geq&\mu(B^{x}_m)\left(1-\frac{\varepsilon}{\ell}\right)~\P[C_{m-1}].
\end{eqnarray*}
Therefore $\P[C_m]>0$.

By reducing the Borel set $B_\ell^{x}$ at the last step of the induction, one can make the measure $\mu^{\otimes \ell}[B_1^x\times...\times B_\ell^x]$ independent of $x \in A^\ell$. Denote by $\beta_\ell$ this measure, then set $B_x=B_1^x\times...\times B_\ell^x$. One gets
\begin{eqnarray*}
\P\big[X_{1:\ell}=x~\big| U_{1:\ell}\in B_x    \big]&=&\frac{\P\big[X_{1:\ell}=x ~;~ U_{1:\ell}\in B_x\big]}{\P\big[U_{1:\ell}\in B_x    \big]}.
\end{eqnarray*}
By independence,
$$
\P[X_{1:\ell}=x,\,U_{1:\ell}\in B_x]\ge\prod_{k=1}^\ell\mu(B_k^x)\left(1-\frac{\varepsilon}{\ell}\right),
\quad\text{ and }\quad
\P[U_{1:\ell}\in B_x]=\prod_{k=1}^\ell\mu(B_k^x),
$$
hence
$$
\P[X_{1:\ell}=x\sach U_{1:\ell}\in B_x]\ge\left(1-\frac{\varepsilon}{\ell}\right)^\ell\geq 1-\varepsilon,
$$
which ends the proof.\hfill $\square$ \end{proof}

\subsection{Approximation until a given time}
\label{S2.3}

Choose $\varepsilon>0$ and $\ell\ge1$ such that $\sum_{n\ge \ell} \eta_n \le \varepsilon$, let $J=[s,t]$ be an interval of integers such that $t-s+1=\ell$.

Then, let $Y$ be a random variable taking values in $A^\ell$, independent of $(X_n,U_n)_{n\in \z}$ and distributed like $X_J$.\\

Lemma~\ref{lemme amorce} provides a real number $\beta_\ell$ and Borel sets $(B_x)_{x \in A^\ell}$, such that
$$\P\big[X_{J}=Y~\big|~U_{J}\in B_{Y}\big] \geq 1-\varepsilon ~\text{ and }~ \P\big[U_{J}\in B_{Y}\big] = \beta_\ell.$$

Using $Y$ and the governing sequence $(U_n)_{n \geq t+1}$, let us construct random variables $(X'_n)_{n \geq s}$ by taking $X'_{J}=Y$ and for every $n> t$
$$X'_n=g(U_n,P'_{n-1}) \text{ where } P'_{n-1}= p(\cdot|X'_{s:n-1}).$$

The random variable $Y$ is useful in the proof of our following result.
\begin{lemme}\label{X tilde a meme loi que X}
For every $n \geq s$, the law of $X'_{s:n}$ is the law of $X_{s:n}$.
\end{lemme}

\begin{proof}
For every $n\geq t+1$, ~$y \in A^{n-s}$ and all $x\in A$,
$$\P\big[X'_n=x~\big|~X'_{s:n-1}=y\big]~=~p(x|y)~=~\P\big[X_n=x~\big|~X_{s:n-1}=y\big].$$
Since the law of $X'_{J}=Y$ is the same as the law of $X_{J}$, the result follows by induction.
\hfill $\square$ \end{proof}

\begin{lemme}\label{lemme maj erreur partie 2} One has
$\P\big[X'\ne X\text{ {\rm on }}[s,+\infty[~\big|~U_{J}\in B_{Y}\big]\leq3\varepsilon.$
\end{lemme}

\begin{proof}
Since $X_n=g(U_n,P_{n-1})$ and $X'_n=g(U_n,P'_{n-1})$, proposition~\ref{prop maj erreur} yields for $n>t$,
$$\P\big[X'_n\ne X_n~\big|~ \F^{X,U}_{n-1}\vee\sigma(Y)\big]\leq 2\|P'_{n-1}- P_{n-1}\|.$$
Let $$(\star)=\P\big[X'_n\ne X_n~;~X'_{s:n-1}=X_{s:n-1}~;~U_{J}\in B_{Y}\big].$$
Since $$\big\{X'_{s:n-1}=X_{s:n-1}~;~U_{J}\in B_{Y}\big\} \in \F^{X,U}_{n-1}\vee\sigma(Y),$$ one gets
$$\begin{array}{rl}
(\star)&\leq \E\big[2\|P_{n-1}- P'_{n-1}\|~  \mathbf{1}_{\{X'_{s:n-1}=X_{s:n-1}\}}~\mathbf{1}_{\{U_{J}\in B_{Y}\}}\big]\\
& \\
&=2\displaystyle\sum_{y\in A^{\ell}\atop z\in A^{n-t-1}}\E\big[\|p(\cdot|X^{\passe}_{s-1}yz)- p(\cdot|yz)\|~\mathbf{1}_{\{yz=X'_{s:n-1}=X_{s:n-1}\}}~ \mathbf{1}_{\{U_{J}\in B_{y}\}}\big]\\
&\leq 2\displaystyle\sum_{y,z}\ \E\big[\|p(\cdot|X^{\passe}_{s-1}yz)- p(\cdot|yz)\|~\mathbf{1}_{\{U_{J}\in B_{y}\}}~\mathbf{1}_{\{yz=X'_{s:n-1}\}}\big],\\
&=2\displaystyle\sum_{y,z}\ \E\big[\|p(\cdot|X^{\passe}_{s-1}yz)-  p(\cdot|yz)\| \big]~ \mu^{\otimes \ell_k}(B_{y}) ~\P\big[X'_{s:n-1}=yz\big]\\
&=2\beta_{\ell} \displaystyle\sum_{x\in A^{n-s}} \E\big[\|p(\cdot|X^{\passe}_{s-1}x)-  p(\cdot|x)\| \big]~ \P\big[X_{s:n-1}=x\big]\\
&=2\beta_{\ell} \eta_{n-s},
\end{array}$$
where the last three equations stem from the independence of $X^{\passe}_{s-1}$, $U_{J}$, $U_{t+1:n-1}$ and $Y$, from lemma~\ref{X tilde a meme loi que X} and from the definition of $\eta_n$.
Hence,
$$\P\big[X'_n\ne X_n~;~X'_{s:n-1}=X_{s:n-1}~\big|~U_{J}\in B_{Y}\big]\leq 2\eta_{n-s},$$
therefore,
$$\P\big[X'_{s:n}=X_{s:n}~|~U_{J}\in B_{Y}\big]\geq \P\big[X'_{s:n-1}=X_{s:n-1}~|~U_{J}\in B_{Y}\big] - 2\eta_{n-s}.$$
By induction, one gets
$$\P\big[X'_{s:n}=X_{s:n}~\big|~U_{J}\in B_{Y}\big]\geq \P\big[X'_{J}=X_{J}~\big|~U_{J}\in B_{Y}\big]-2\displaystyle\sum_{m=\ell}^{n-s}\eta_m.
$$
Since $X'_{J}=Y$ and $\P[X_J=Y ~|~U_J \in B_Y] \geq 1-\varepsilon$, this yields
$$
\P\big[X'=X\text{ on }[s,+\infty[~\big|~U_{J}\in B_{Y}\big]\ge1-\varepsilon-2\sum_{m=\ell}^{\infty}\eta_m \geq 1-3\varepsilon,$$
which ends the proof. \hfill $\square$ \end{proof}

\subsection{Successive approximations and end of the proof of theorem~\ref{theo e}}
\label{S2.4}

Our next step in the proof of theorem~\ref{theo e} is to approach the random variable $X_0$ by measurable functions of the governing sequence. To this aim, we group the innovations by intervals of times.
For every $m>0$ one chooses $L_m$ such that $$\sum_{n\geq L_m} \eta_n \leq 1/m.$$

For each $m$, lemma~\ref{lemme amorce} (the priming lemma) applied to $\ell=L_m$ and $\varepsilon=1/m$ provides a real number $\beta_{L_m}>0$ and Borel sets $(B_x)_{x\in A^{L_m}}$ of $S^{L_m}$ with measure $\beta_{L_m}$ such that $$\P \big[X_{1:L_m}=x~|~U_{1:L_m}\in B_x\big]\geq 1-1/m.$$
Choose an integer $M_m\geq 1/\beta_{L_m}$.
Split $\z^{*}_-$ into $M_1$ intervals of length $L_1$, $M_2$ intervals of length $L_2$, ... More precisely set, for every $n\geq 1$,
$$
\ell_n=L_m \ \text{ if } \ \ M_1+\cdots+M_{m-1}<n\leq M_1+\cdots+M_m\\
$$
and
$$
\varepsilon_n=1/m \ \text{ if }\ \  M_1+\cdots+M_{m-1}<n\leq M_1+\cdots+M_m\\
$$

Therefore, for every $k\geq0$ one gets
$$
\sum_{n\geq \ell_k}\eta_n \leq \varepsilon_k.
$$
At last, for every $k\geq0$, set  $$t_k=-\sum_{1\leq n\leq k}\ell_n$$ that is to say $t_0=0$ and $t_k=t_{k-1}-\ell_{k}$ for $k \geq 1$.
Define, for $k \geq 0$,  the interval of integers $$J_k=[t_k,t_k+\ell_k-1]=[t_k,t_{k-1}-1] \text{ and } X_{J_k}=X_{t_k:t_{k-1}-1}.$$

\eject
\newpage

\begin{figure}[h]
\begin{center}
\includegraphics[width=.95\textwidth]{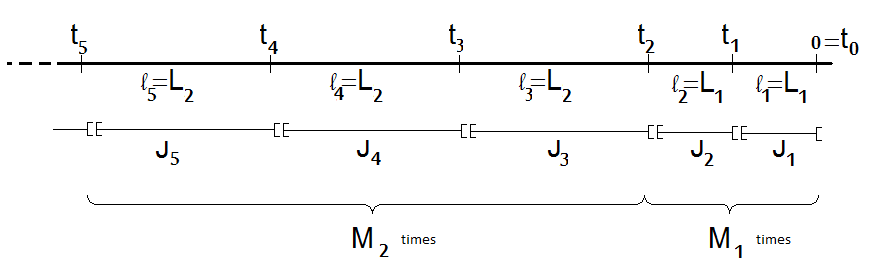}
\caption{Splitting $\z^*_-$ in intervals of times.}
\end{center}
\label{Graphique innovations 3}
\end{figure}

Let $Y=(Y_k)_{k\geq1}$ be a sequence of random variables, independent of $(X_n,U_n)_{n\in \z}$ and such that for every $k\geq 1$, the law of $Y_k$ is the law of $X_{1:\ell_k}$.

For every $k\geq 0$, let us use the construction of section~\ref{S2.3}: set $X^k_{J_k}=Y_k$, then for every $n\geq t_k+\ell_k=t_{k-1}$
$$X^k_n=g(U_n,P^k_{n-1}) \text{ where } P^k_{n-1}= p(\cdot|X^k_{t_k:n-1}).$$

Therefore lemma~\ref{lemme maj erreur partie 2} yields the inequality
 $$\P\big[X_{t_k:0}\ne X^k_{t_k:0}~\big|~U_{J_k}\in B_{Y_k}\big]\leq3\varepsilon_k,$$

and $$\P\big[X_0^k \ne X_0 ~\big|~ U_{J_k} \in B_{Y_k} \big] \to 0\text{ when } k \to +\infty.$$

Moreover each event $\big\{U_{J_k}\in B_{Y_k}\big\}$ is independent of the others (indeed they are functions of random variables $U_k$ for disjoint sets of indices $k$) and
$$\sum_{k\geq 1} \P\big[U_{J_k}\in B_{Y_k} \big]=\sum_{k\geq 1} \beta_{\ell_k} =\sum_{m=1}^{+\infty}M_m \beta_{L_m} =+\infty$$
since $M_m \beta_{L_m}\ge 1$  by choice of $M_m$.

Lemma~\ref{lem2.1}, stated below, provides a deterministic increasing function $\theta$ such that
$$\sum_{k\geq 1} \P\big[X_0^{\theta(k)}\ne X_0 ~;~ U_{J_{\theta(k)}} \in B_{Y_{\theta(k)}} \big] < +\infty$$
$$\text{and } \sum_{k\geq 1} \P\big[ U_{J_{\theta(k)}} \in B_{Y_{\theta(k)}}\big] = +\infty.$$
Using Borel-Cantelli's lemma, one deduces that\\
\begin{itemize}
\item $\big\{X_0^{\theta(k)}\ne X_0\big\} \cap \big\{ U_{J_{\theta(k)}} \in B_{Y_{\theta(k)}} \big\} \text{ is realized for a finite number of }k\text{ only, a.s.}$
\item $\big\{ U_{J_{\theta(k)}} \in B_{Y_{\theta(k)}} \big\} \text{ is realized for an infinite number of } k \text{ a.s.}$.
\end{itemize}
Thus, for every $a \in A$,
$$\{X_0=a\}=\limsup_{k \to \infty} \big\{ U_{J_{\theta(k)}} \in B_{Y_{\theta(k)}} \big\}\cap \big\{X_0^{\theta(k)}=a\big\}.$$
Therefore, $\{X_0=a\}$ belongs to $\F_0^U \vee \sigma(Y)$. Since the sequence $Y=(Y_k)_{k \geq 0}$ is independent of $\F_0^{U,X}$, one gets
$$\P\big[X_0=a\big|\F_0^U \big]=\P\big[X_0=a~\big|~\F_0^U\vee\sigma(Y) \big]=\mathbf{1}_{\{X_0=a\}} \text{ a.s.},$$
therefore $\{X_0=a\}\in \F_0^U$. By stationarity of the process $(X,U)$, one gets the inclusion of the filtration $\F^X$ into the filtration $\F^U$. Therefore lemma~\ref{immersion} yields that $\F^X$ is immersed in $\F^U$, which ends the proof. $\hfill \square~$

\begin{lemme}\label{lem2.1}
Let $(a_n)_{n\ge0}$ and $(b_n)_{n\ge0}$ denote two bounded sequences of nonnegative real numbers such that the series $\displaystyle\sum_nb_n$ diverges and such that $a_n\ll b_n$. Then there exists an increasing function $\theta : \N \to \N$ such that the series
$\displaystyle\sum_{n} a_{\theta(n)}$ converges and the series $\displaystyle\sum_{n} b_{\theta(n)}$ diverges.
\end{lemme}

\section{Examples}\label{S3}

In this section we study some examples showing the advantages and the limitations of our results.
\begin{itemize}
\item Our first example (section~\ref{S3.1}) is a chain with memory of variable length which fulfills the hypotheses of theorem~\ref{theo e} but not those of theorem~\ref{theo gamma} nor theorem~\ref{theo delta}. Its natural filtration is standard.

\item Our second example (section~\ref{S3.2}) is derived from the well known $[T,T^{-1}]$ transformation. It provides a stationary process with values in a finite space, whose natural filtration is not standard. This example does not fulfills any of the two conditions of theorem~\ref{theo e} (namely the priming condition and the summability of the gaps).

\item Our third example (section~\ref{S3.3}) is a slight adaptation of the second one, where the filtration of the process is still non-standard although the priming condition is fulfilled.

\item Our fourth and last example (section~\ref{S3.4}) is another adaptation of the second example in which the filtration is standard although the condition of summability of the gaps is not fulfilled and the related conditional probabilities are close to those of the second example.
\end{itemize}
\subsection{First example : Parity of the number of $1$ in a row}\label{S3.1}

This example provides a setting where one proves standardness using theorem~\ref{theo e}.

Let $(X_n)_{n \in \z}$ be a stationary process taking values in $\{0,1\}$ such that
$$ \P[X_0=0 ~|~ \F_{-1}^X]=\frac{1}{3}+\frac{1}{3}\mathbf{1}_{\{T \text{ is even or }T=-\infty\}} ~~~~\text{where}~~~~ T:=\sup\{k<0~:~X_k=0\}.$$

The existence of such a process is ensured by proposition 2.10 in~\cite{galves2008stochastic}. A simple computation gives, for every $n \geq 0$,
$$\gamma_n=\frac{1}{2},~~ \delta_n=\frac{1}{3},~~ \alpha_n=\frac{1}{3},~~ \eta_n\leq\frac{1}{3}\P[T<-n]\leq \left(\frac{2}{3}\right)^{n+1}.$$
Therefore this process fulfills the hypotheses of theorem~\ref{theo e} (and its corollary~\ref{memoire variable}) but neither those of theorem~\ref{theo gamma} nor those of theorem~\ref{theo delta}. The filtration $\F^X$ is standard.

\subsection{Second example : Random walk in random scenery}\label{S3.2}

The following is a process whose filtration is not standard.

Let $X=(X_n)_{n \in \z}$ and $C=(C_s)_{s \in \z}$ be two independent sequences of i.i.d. random variables with uniform law on $\{-1,1\}$. Set
$$\begin{array}{l}
S_n=X_1+\cdots+X_n \text{ if } n \geq 0\\
S_n=-X_{n+1}-\cdots-X_0 \text{ if } n<0
\end{array}.$$
Therefore $S_{n+1}=S_n+X_{n+1}$ for every $n \in \z$.
Set $C_{S_n}=Y_n$. The stationary process $Z$ defined by $Z_n=(X_n,Y_n)$ for every $n\in \z$ and taking values in $A=\{-1,1\}^2$ is called random walk in random scenery.

This process is derived from the process $((X_{n+\cdot},C_{S_n+\cdot}))_{n \in\z}$ where $X_{n+\cdot}=(X_{n+m})_{m\leq0}$ is the trajectory of $X$ until time $n$ and $C_{S_n+\cdot}=(C_{S_n+s})_{s\in \z}$ is the scenery seen from $S_n$. It is easy to prove that the processes $Z$ and $(X_{n+\cdot},C_{S_n+\cdot})_{n \in \z}$ generate the same filtration. Indeed, given $(X_k,C_{S_k})$ for every $k\leq n$, one knows the trajectory $X_{n+\cdot}$ and one can deduces the increments $(S_{n}-S_k)_{k \leq n}$. Since those increments visit almost surely every integer, one can recover the scenery seen from $S_n$.

The process $((X_{n+\cdot},C_{S_n+\cdot}))_{n \in\z}$ is the most famous $[T,T^{-1}]$ process. Indeed the $[T,T^{-1}]$ transformation is the application from $\{-1,1\}^{\z_-}\times\{-1,1\}^{\z}$ into itself defined by
$$[T,T^{-1}]((x_{n})_{n \leq 0},(c_s)_{s \in \z})=((x_{n-1})_{n \leq 0},(c_{s-x_0})_{s \in \z}).$$
One checks that for every $n\leq 0$
$$(X_{n+\cdot},C_{S_n+\cdot})=[T,T^{-1}]^{-n}(X,C).$$

According to~\cite{heicklen-t}, the natural filtration of the process $((X_{n+\cdot},C_{S_n+\cdot}))_{n \in\z}$ is not standard though its asymptotic $\sigma$-field at $-\infty$ is trivial. Therefore the same holds for the natural filtration of $Z$.

Let $n \geq 0$. Let us study the probabilities $p(a|z)$ for $a \in A$ and $z \in A^n$. Note $z=(z_{-n},...,z_{-1})$, $z_k=(x_k,y_k)$ and $a=(x_0,y_0)$.

First, note that some of the events
$$\Big\{Z_{-n:-1}=z\Big\}=\Big\{X_{-n:-1}=x_{-n:-1}~;~Y_{-n:-1}=y_{-n:-1}\Big\}$$
are impossible. Indeed, by definition of the process $Y$ for $i<j$, $Y_i=Y_j$ on the event $X_{i+1}+\cdots+X_j=0$. When the event $\{Z_{-n:-1}=z\}$ is impossible, one says that $z$ is not admissible. Note that to compute $\eta_n$, one only needs to consider probabilities $p(\cdot|z)$ and $p(\cdot|wz)$ for admissible $z \in A^n$ and $w \in A^\passe$. Yet, $wz$ may be non-admissible even if $z$ and $w$ are admissible.

Assume that $z$ is admissible. Then $\P\big[X_0=x_0|Z_{-n:-1}=z\big]=1/2$. If for some $i \in \{-n,...,-1\}$, $x_{i+1}+\cdots+x_0=0$, then the conditions $Z_{-n:-1}=z$ and $X_0=x_0$ imply that $Y_0=y_i$. Otherwise, the color $Y_0$ is independent of $Z_{-n:-1}$ and $X_0$.
Thus for any admissible word $z\in A^n$,
$$p(a|z)=\left| \begin{array}{ll} 1/2~~ & \text{ if there exists } i \text{ such that }x_{i+1}+\cdots+x_0=0 \text{ and }y_0=y_i,\\
0 &\text{ if there exists } i \text{ such that }x_{i+1}+\cdots+x_0=0 \text{ and }y_0\ne y_i,\\
1/4 &\text{ otherwise.} \end{array} \right.
$$
Therefore for every $n \geq0$, $\gamma_n=1$ and $\delta_n=\alpha_n=1/2$.

Furthermore, for almost every admissible word $w=(x_n,y_n)_{n<0}$ in $A^\passe$, there exists $t<0$ such that $x_{t+1}+\cdots+x_0=0$ and the same argument gives that $$p(a|w)=\frac{1}{2}{\mathbf{1}}_{\{y_t=y_0\}}.$$

For non-admissible $w \in A^\passe$, the value of $p(a|w)$ can be chosen arbitrarily. Set  $$p(a|w)=\frac{1}{2}{\mathbf{1}}_{\{y_d=y_0\}} \text{ where } d=\sup\{t\leq-1~:~x_{t+1}+\cdots+x_0=0\}$$ if $d$ is well defined, and $p(a|w)=1/4$ otherwise.

With this convention, one gets that, for almost any admissible $w\in A^\passe$ and $z=((x_{-n},y_{-n}),\cdots,(x_{-1},y_{-1})) \in A^n$,
$$
|p(a|z)-p(a|wz)|=\left|\begin{array}{l}0\text{ if there exists }t \in [-n,-1]\text{ such that }x_t+\cdots+x_0=0\\
1/4\text{ otherwise.}
\end{array}\right.$$

Therefore, for every $n \geq0$,
\begin{eqnarray*}
\eta_n&=&\frac{1}{4}\Big(\P\big[\forall i \in \{-n,...,-1\}, X_{i+1}+\cdots+X_{-1}\geq0\big]\\
& &~~~~~~~~~~+\P\big[\forall i \in \{-n,...,-1\}, X_{i+1}+\cdots+X_{-1}\leq0\big]\Big)\\
&=&\frac{1}{2}~\P\big[\forall i \in \{-n,...,-1\}, X_{i+1}+\cdots+X_{-1}\geq0\big].
\end{eqnarray*}
One sees that $\eta_n \sim C/\sqrt{n}$ with $C \in \R^+_*$. Hence the process $Z$ does not fulfill any of the hypotheses of theorem~\ref{theo e}.\\

\subsection{Third example : Random walk in random scenery with misreading}\label{S3.3}

We construct a variant of the random walk in random scenery which fulfills the priming condition but whose natural filtration is not standard.

Construct $Z_n=(X_n,Y_n)$ as in section~\ref{S3.2}. Fix $q \in ]0,1/2[$. Let $(\xi_n)_{n \in \z}$ be a sequence of i.i.d. random variables taking values in $\{-1,1\}$, independent of $\F^Z$ and such that $\P[\xi_0=1]=1-q$. Define a process $(Z'_n)_{n \in \z}$ by
$$Z'_n=(X_n,Y_n \xi_n).$$
The process $(Z'_n)_{n \in \z}$ is a random walk in random scenery in which at each time, one misreads the color of the site $Y_n$ with probability $q$.

The processes $Z'$ and $(Z_n,\xi_n)_{n \in \z}$ generate the same filtration. Indeed, the random variables $\xi_m$ associated to the times $m<n$ where $S_m=S_n$ are independent and take the value $1$ with probability $1-q>1/2$, therefore the color $Y_n$ is the most common color among the colors $Y_m \xi_m$ seen at those times. Therefore, for almost every $z \in A^\passe$ and $a\in A$, the corresponding conditional probability $p'(a|z)$ is  equal to $q/2$ or to $(1-q)/2$ depending on these colors.

Moreover, by independent enlargement, $\F^Z$ is immersed into $\F^{Z,\xi}=\F^{Z'}$. Since $\F^Z$ is non-standard, one deduces that $\F^{Z'}$ is not standard either.

By a short calculation, one gets for every $n > 0$, $$\gamma_n=\frac{q}{1-q},\quad \delta_n=\frac{1-2q}{4}, \quad \alpha_n=1-2q.$$
Since the probabilities $p'(a|z)$ related to this process satisfy $p'(a|z)\geq q/2$, for every $a \in A$ and $z \in A^{\passe}$, the priming condition is fulfilled. The exact value of $p'(a|z)$ for $z \in A^n$ is difficult to compute, but the corresponding gaps $\eta'_n$ verify
\begin{eqnarray*}
\eta'_n&\ge&\left(\frac{1}{4}-\frac{q}{2}\right)\Big(\P\big[\forall i \in \{-l,...,-1\}, X_{i+1}+\cdots+X_{-1}\geq0\big]\\
& &~~~~~~~~~~+\P\big[\forall i \in \{-l,...,-1\}, X_{i+1}+\cdots+X_{-1}\leq0\big]\Big)\\
&=&(1-2q)\eta_n.
\end{eqnarray*}
Therefore for $q<1/2$, the sequence $(\eta'_n)_n$ is not summable, thus the process $Z'$ does not verify the condition of summability of the gaps.

\subsection{Fourth example : Random walk in renewed random scenery}\label{S3.4}

We construct another variant of the random walk in random scenery in which the natural filtration is standard although the condition of summability of the gaps of theorem~\ref{theo e} is not fulfilled.

We consider a variant of the process $(X_n,C_{S_n+\cdot})_{n \in \z}$ in which at each time $n$ the color at $0$ of the scenery seen from $S_n$ is changed with probability $q \in ]0,1/2[$.
For every $g \in \{-1,1\}^{\z}$, denote $\overline{g}\in \{-1,1\}^{\z}$ the application defined by $$\overline{g}(s)=g(s) \text{ for } s\ne 0 \text{ and } \overline{g}(0)=-g(0).$$ Let $(X_n,G_n)$ be a stationary Markov chain with values in $\{-1,1\}\times\{-1,1\}^{\z}$, with transition probabilities
$$p\big((x,g),(x',g')\big)=\left|\begin{array}{ll}(1-q)/2 &\text{if }g'=g(x'+\cdot)\\
 & \\
                                         q/2 & \text{if }g'=\overline{g}(x'+\cdot)   \end{array} \right. .$$
The random walk in renewed random scenery is the process $Z''=(Z''_n)_{n \in \z}$ defined by $Z''_n=(X_n,G_n(0))$.

The corresponding probabilities $p''(a|z)$ are close to the probabilities $p(a|z)$.
Indeed,
$$p''(a|z)=\left|\begin{array}{ll} (1-q)/2 & \text{if }p(a|z)=1/2 \\
 & \\
                                    q/2 & \text{if }p(a|z)=0  \\
 & \\
                                    1/4 &\text{if }p(a|z)=1/4\end{array} \right. .$$
Therefore the corresponding gaps verify $$\eta''_n=(1-2q)\eta_n.$$

To show that the filtration $\F^{Z''}$ is standard, one can use the following trick: instead of changing the color at $0$ of the scenery $G_n$ with probability $q$, one draws at random this color with probability $2q$. One needs a random variable $\varepsilon_n$ taking the value $1$ if this drawing occurs and $0$ otherwise, and a random variable $\kappa_n$ giving the color obtained if the drawing occurs.

To construct these random variables, consider two independent sequences of random variables $(\beta_n)_{n \in \z}$ and $(V_n)_{n \in \z}$, independent of $\F^{X,G}$ such that
\begin{itemize}
\item the $\beta_n$ are i.i.d. Bernoulli variables of parameter $(1-2q)/(1-q)$,
\item the $V_n$ are i.i.d. and uniform on $\{-1,1\}$.
\end{itemize}
Let, for every $n \in \z$,
$$\varepsilon_n=1-\beta_n \textbf{1}_{\{G_n(0)=G_{n-1}(X_n)\}} \ \ \text{ and }\ \  \kappa_n=G_n(0)\textbf{1}_{\{\varepsilon_n=1\}}+V_n \textbf{1}_{\{\varepsilon_n=0\}}.$$

Let us show that the random variables $U_n=(X_n,\varepsilon_n,\kappa_n)$ constitute a governing sequence for the process $Z''$. Given $Z''_{n-1}$ and $U_n$, one deduces $Z''_n$ thanks to the equalities
$$\begin{array}{ll}
G_n(s)=G_{n-1}(X_n+s) & \text{ if } \varepsilon_n=0 \text{ or } s \ne 0\\
G_n(0)=\kappa_n & \text{ if } \varepsilon_n=1.
\end{array}$$
It remains to check that $U_n$ is independent of the $\sigma-$field $\G_{n-1}=\F^{X,G,\beta,V}_{n-1}$ and \textit{a fortiori} of $\F^{Z'',U}_{n-1}$. Thanks to the independence of the processes $\beta$, $V$ and $(X,G)$ one gets for every $x\in \{-1,1\}$,
$$\begin{array}{l}\P\big[X_n=x~;~G_n=G_{n-1}(X_n+\cdot)\big|\G_{n-1}\big]=(1-q)/2,\\
 \\
        \P\big[X_n=x~;~G_n\ne{G}_{n-1}(X_n+\cdot)\big|\G_{n-1}\big]=q/2. \end{array}$$
Therefore, for every $c$ and $x$ in $\{-1,1\}$,
$$\begin{array}{rl}
\P[\varepsilon_n=1~;~\kappa_n=c~~;~~X_n=x ~|~ \G_{n-1}]
&=\P[\varepsilon_n=1~;~G_n(0)=c~;~ X_n=x ~|~\G_{n-1}]\\
&=(1)+(2)+(3),
\end{array}$$
with
$$
\begin{array}{l}
(1)=\P\big[\beta_n=0~;~G_n\ne{G}_{n-1}(X_n+\cdot)~;~G_{n-1}(X_n)=-c~;~ X_n=x ~|~\G_{n-1}\big],\\
 \\
(2)=\P\big[\beta_n=0~;~G_n=G_{n-1}(X_n+\cdot)~;~G_{n-1}(X_n)=c~;~ X_n=x ~|~\G_{n-1}\big],\\
 \\
(3)=\P\big[\beta=1~;~G_n\ne{G}_{n-1}(X_n+\cdot)~;~G_{n-1}(X_n)=-c~;~ X_n=x ~|~\G_{n-1}\big].
\end{array}
$$
One gets
\begin{eqnarray*}
(1)&=&\displaystyle\frac{q}{1-q}\times \P\big[G_n\ne{G}_{n-1}(X_n+\cdot)~;~G_{n-1}(x)=-c~;~ X_n=x ~|~\G_{n-1}\big]\\
&=&\displaystyle\frac{q}{2}\times\displaystyle\frac{q}{1-q}\times\P\big[G_{n-1}(x)=-c ~|~\G_{n-1}\big]\\
&=&\displaystyle\frac{q}{2}\times\displaystyle\frac{q}{1-q}\times\mathbf{1}_{\{G_{n-1}(x)=-c\}},
\end{eqnarray*}

\begin{eqnarray*}
(2)&=&\displaystyle\frac{q}{1-q}\times\P\big[G_n=G_{n-1}(X_n+\cdot)~;~G_{n-1}(x)=c~;~ X_n=x ~|~\G_{n-1}\big]\\
&=&\displaystyle\frac{1-q}{2}\times\displaystyle\frac{q}{1-q}\times\P\big[G_{n-1}(x)=c~|~\G_{n-1}\big]\\
&=&\displaystyle\frac{q}{2}\times\mathbf{1}_{\{G_{n-1}(x)=c\}},
\end{eqnarray*}
and
\begin{eqnarray*}
(3)&=&\displaystyle\frac{1-2q}{1-q}\times\P\big[G_n\ne{G}_{n-1}(X_n+\cdot)~;~G_{n-1}(x)=-c~;~ X_n=x ~|~\G_{n-1}\big]\\
&=&\displaystyle\frac{q}{2}\times\displaystyle\frac{1-2q}{1-q}\times\P\big[G_{n-1}(x)=-c ~|~\G_{n-1}\big]\\
&=&\displaystyle\frac{q}{2} \times\displaystyle\frac{1-2q}{1-q}\times\mathbf{1}_{\{G_{n-1}(x)=-c\}}.
\end{eqnarray*}
Thus, for every $c$ and $x$ in $\{-1,1\}$,
$$\P[\varepsilon_n=1~;~\kappa_n=c~~;~~X_n=x ~|~ \G_{n-1}]=\displaystyle\frac{q}{2}.$$
Moreover, by independence of $\beta_n$, $V_n$ and $\G_{n-1}$,
\begin{eqnarray*}
\P\big[\varepsilon_n=0&;& \kappa_n=c~;~ X_n=x~|~\G_{n-1}\big]\\
&=&\P\big[\beta_n=1~;~ G_n=G_{n-1}(X_n+\cdot)~;~ V_n=c~;~ X_n=x|\G_{n-1}\big]\\
&=&\frac{1-2q}{1-q}\times\frac{ 1}{2}\times\frac{1-q}{2}\\
&=&\frac{1-2q}{4}.
\end{eqnarray*}
This shows that the random variables $U_n=(X_n,\varepsilon_n,\kappa_n)$ constitute a governing sequence for the process $Z''$.

Let us show the inclusion $\F^{Z''}_n\subset \F^U_n$ for any $n \in \z$, that is to say, that the sequence $(U_k)_{k \leq n}$ is sufficient to recover the scenery $G_n$ seen from $S_n$. The variables $(X_k)_{k\leq n}$ determine the increments $(S_n-S_k)_{k\leq n}$ and for every $s \in \z$, $S_n-S_k=s$ for an infinite number of times $k \leq n$. Among those times, there is an infinite number of times such that $\varepsilon_k=1$. The value of $\kappa_k$ at the last time $k \leq n$ such that $S_n-S_k=s$ and $\varepsilon_k=1$ is equal to $G_n(s)$. Therefore $\F^{G} \subset \F^U$, and since $\F^X\subset\F^U$, one gets  $\F^{Z''} \subset \F^U$.
Finally, lemma~\ref{immersion} yields that $\F^{Z''}$ is immersed in $\F^U$, therefore the natural filtration of the process $(Z''_n)_{n \in \z}$ is standard.

\section{Proofs of auxiliary facts}
\label{ss.profs}
\subsection{Inequalities involving $\alpha_n, \delta_n, \gamma_n$ and $\eta_n$}\label{S4.1}


{\bf To prove that $ \delta_n\le\gamma_n$ for every $n \geq 0$,} consider $x$ and $y$
in $A^{\passe}$ and $z \in A^n$. Then,
\begin{align*}
\|p(\cdot|xz)-p(\cdot|yz)\|&=\sum_{a\in A}\big[p(a|xz)-p(a|yz)\big]_+\\
&=\sum_{a\in A}p(a|xz)\left(1-\frac{p(a|yz)}{p(a|xz)}\right)_+
\le\sum_{a\in A}p(a|xz)\gamma_n=\gamma_n.
\end{align*}
Taking the supremum over $x$, $y$ and $z$, one gets $\delta_n\leq \gamma_n$.
~\\


{\bf
To prove that $\eta_n \le \delta_n$ for every $n \geq 0$,
} consider for every $z \in A^{n}$,  the law $Q_z$ of $X^{\passe}_{-n-1}$ conditionally on $X_{-n:-1}=z$.
Then,
$$
p(\cdot|z)=\int_{A^{\passe}}p(\cdot|yz)Q_z({\rm d}y).
$$
Thus, for every $x$ in $A^{\passe}$ and $z$ in $A^{n}$,
\begin{eqnarray*}
\|p(\cdot|z)-p(\cdot|x z)\|
&=&\Big\|\int_{y\in A^\passe}\big( p(\cdot|yz)-p(\cdot|xz)\big)Q_z({\rm d}y)\Big\|
\\
&\le&\int_{y\in A^\passe}\big\| p(\cdot|yz)-p(\cdot|xz)\big\|Q_z({\rm d}y)
\\
&\le&\sup_{y\in A^\passe}\big\| p(\cdot|yz)-p(\cdot|xz)\big\|
\le\delta_n.
\end{eqnarray*}
For every $z$ in $A^{n}$,
$\big\|p(\cdot|z)-p(\cdot|X^{\passe}_{-n-1}z)\big\|\leq \delta_n$ almost surely. Taking the expectation and the average over $z$, one gets $\eta_n \le \delta_n$.

~\\

{\bf
To prove that $\delta_n \le \alpha_n$ for every $n \geq 0$,
}consider $z \in A^n$ and $y,y'\in A^\passe$. Then,
\begin{eqnarray*}
\|p(\cdot|yz)-p(\cdot|y'z)\|&=&\sum_{a \in A}\big|p(a|yz)-p(a|y'z)\big|_+\\
&=&\sum_{a \in A}\Big(p(a|yz)-\min\big(p(a|yz),,p(a|y'z)\big)\Big)\\
&\leq&1-\inf_{z\in A^n}\sum_{a \in A}\inf\big\{p(a|yz)~:~y \in A^\passe\big\}\\
&=&\alpha_n.
\end{eqnarray*}
This ends the proof.

\subsection{Proof of lemma~\ref{immersion}}\label{S4.2}

Assume that $X$ is a process valued in a measurable space $(E,\mathfrak{E})$ and that $U$ is a governing and generating process of $X$.
Let $n\in\z$.
Since $U$ governs $X$, there exists a measurable function $\psi_n$ such that $X_{n+1}=\psi_n(U_{n+1},X^{\passe}_n)$ (axiom (ii)).
Let $B\in \mathfrak{E}$. We try to estimate
$$
\rho=\P[X_{n+1}\in B\sach\F^{X,U}_n]=\P[\psi_n(U_{n+1},X^{\passe}_n)\in B\sach\F^{X,U}_n].
$$
Since $U$ governs $X$,
$U_{n+1}$ and $\F^{X,U}_n$ are independent (axiom (i)) hence $\rho$ is a function of $X^{\passe}_n$ only, that is,
$$
\rho=\P[\psi(U_{n+1},X^{\passe}_n)\in B\sach\F^{X}_n]=\P[X_{n+1}\in B\sach\F^{X}_n].
$$
Hence $\F^X_{n+1}$ is independent of $\F^U_n$ conditionally on $\F^X_n$. This shows that $\F^X$ is immersed in $\F^U$.

~\\
\subsection{Proof of corollary~\ref{memoire variable}}\label{S4.3}
Assume that $X$ is a chain with memory of variable length and $Y$ be an independent copy of $X$. As
$
p(\cdot|X^\passe_{-n-1}Y_{-n:-1})=p(\cdot|Y_{-n:-1}).
$
on the event $\{\ell(Y^\passe_{-1}) \le n \}$,
$$\|p(\cdot|X^\passe_{-n-1}Y_{-n:-1})-p(\cdot|Y_{-n:-1})\|\le\un_{\{\ell(Y^\passe_{-1})\ge n+1\}}.$$
Taking expectations, one gets $\eta_n \le\P[\ell(Y^\passe_{-1})\ge n+1]$, hence
$$
\sum_{n\ge0} \eta_n \le\E[\ell(Y^\passe_{-1})]<+\infty.
$$
This ends the proof.

\subsection{Proof that $\mathcal{H}(\gamma)$ provides a positive lower bound for $p(a|x)$}\label{S4.4}

We show that $\mathcal{H}(\gamma)$ implies the existence of $c>0$ such that $p(a|x)\ge c$ for every $x$ in $A^{\passe}$ and $a$ such that $\P[X_0=a]>0$.\\
Assume that $\mathcal{H}(\gamma)$, that is $$\sum_{k=0}^{+\infty} \prod_{n=0}^{k}(1-\gamma_n)=+\infty.$$
Therefore, $1-\gamma_0>0$. By definition of $\gamma_0$, for every $a \in A$, $x,y \in A^\passe$,
$$p(a|x)\geq(1-\gamma_0)p(a|y).$$ Integrating this inequality with respect to the law of $X^\passe$, one gets
$$p(a|x)\geq(1-\gamma_0)\P[X_0=a].$$ Since $A$ is finite, this ends the proof.


\bibliographystyle{plain}
\bibliography{bibliFr}

\end{document}